\documentclass[10pt,leqno]{amsart}

\usepackage{amsmath}
\usepackage{amsfonts}
\usepackage{amssymb}
\usepackage{graphicx}
\usepackage{color}
\usepackage{hyperref}
\usepackage{xcolor}

\newtheorem{theorem}{Theorem}[section]
\newtheorem*{theorem*}{Theorem}

\newtheorem{definition}[theorem]{Definition}

\newtheorem{conjecture}[theorem]{Conjecture}

\newtheorem{remark}[theorem]{Remark}

\newcommand{\R}{\mathbb{R}}

\def\Ric{\text{Ric}}

\def\a{\alpha}
\def\l{\lambda}

\def\R{\mathbb{R}}

\def\vp{\varphi}

\def\Ric{\operatorname{Ric}}

\def\tr{\operatorname{tr}}

\numberwithin{equation}{section}
\makeatletter
\newcommand*\owedge{\mathpalette\@owedge\relax}
\newcommand*\@owedge[1]{%
  \mathbin{%
    \ooalign{%
      $#1\m@th\bigcirc$\cr
      \hidewidth$#1\m@th\wedge$\hidewidth\cr
    }%
  }%
}
\makeatother

\begin{document}

\title[curvature operator of the second kind]{Manifolds with $4\frac{1}{2}$-positive curvature operator of the second kind}

\dedicatory{Dedicated to Professor Peter Li on the occasion of his 70th birthday.}

\author{Xiaolong Li}
\address{Department of Mathematics, Statistics and Physics, Wichita State University, Wichita, KS, 67260}
\email{xiaolong.li@wichita.edu}



\begin{abstract}
We show that a closed four-manifold with $4\frac{1}{2}$-positive curvature operator of the second kind is diffeomorphic to a spherical space form. The curvature assumption is sharp as both $\mathbb{CP}^2$ and $\mathbb{S}^3 \times \mathbb{S}^1$ have $4\frac{1}{2}$-nonnegative curvature operator of the second kind. 
In higher dimensions $n\geq 5$, we show that closed Riemannian manifolds with $4\frac{1}{2}$-positive curvature operator of the second kind are homeomorphic to spherical space forms. 
These results are proved by showing that $4\frac{1}{2}$-positive curvature operator of the second kind implies both positive isotropic curvature and positive Ricci curvature. Rigidity results for $4\frac{1}{2}$-nonnegative curvature operator of the second kind are also obtained.  
\end{abstract}

\maketitle

\section{Introduction}
In 1986, Nishikawa \cite{Nishikawa86} conjectured that a closed Riemannian manifold with positive (respectively, nonnegative) curvature operator of the second kind is diffeomorphic to a spherical space form (respectively, Riemannian locally symmetric space). 
Recall that the curvature tensor acts on the space of symmetric two-tensors $S^2(T_pM)$ via  
\begin{equation}\label{eq 1.1}
    \mathring{R}(e_i \odot e_j) =\sum_{k,l=1}^n R_{iklj} e_k \odot e_l,
\end{equation}
where $\{e_1, \cdots, e_n\}$ is an orthonormal basis of the tangent space $T_pM$ at $p$ and $\odot$ denotes the symmetric product. 
In Nishikawa's conjecture, curvature operator of the second kind, denoted by $\mathring{R}$ in this paper, refers to the bilinear form 
$$\mathring{R}:S^2_0(T_pM) \times S^2_0(T_pM) \to \R$$ obtained by restricting the symmetric linear map $\mathring{R}:S^2(T_pM) \to S^2(T_pM)$ defined in 
\eqref{eq 1.1} 
to the space of traceless symmetric two-tensors $S^2_0(T_pM)$.
We refer the reader to Section 2 for a detailed discussion on curvature operator of the second kind.

Recently, the positive case of Nishikawa's conjecture was proved by Cao, Gursky and Tran \cite{CGT21} and the nonnegative case was settled by the author \cite{Li21}. In fact, the assumption can be weakened. 
\begin{theorem}\label{thm 3+}
Let $(M^n,g)$ be a closed Riemannian manifold of dimension $n\geq 3$. 
\begin{enumerate}
    \item  If $M$ has three-positive curvature operator of the second kind, then $M$ is diffeomorphic to a spherical space form;
    \item If $M$ has three-nonnegative curvature operator of the second kind, then $M$ is either flat, or diffeomorphic to a spherical space form, or isometric to a quotient of a compact irreducible symmetric space\footnote{After the submission of this paper, Nienhaus, Petersen and Wink \cite{NPW22} proved that a closed $n$-dimensional Riemannian manifold with $\lfloor \frac{n+2}{2} \rfloor$-nonnegative curvature operator of the second kind is either a rational homology sphere or flat, thus ruling out the third possibility here.}.
\end{enumerate}
\end{theorem}

The key observation is that if the curvature operator of the second kind of $M$ is three-positive (respectively, three-nonnegative), then $M\times \R$ has positive (respectively, nonnegative) isotropic curvature.
This was first proved in \cite{CGT21} under the stronger assumption of two-positivity/two-nonnegativity of the curvature operator of the second kind and was then weakened to three-positivity/three-nonnegativity in \cite{Li21}. 
Part (1) of Theorem \ref{thm 3+} then follows from a result of Brendle \cite{Brendle08}, which asserts that  the normalized Ricci flow evolves an initial metric satisfying the property that $M\times \R$ has positive isotropic curvature into a limit metric with constant sectional curvature, thus implying that $M$ must be diffeomorphic to a spherical space form. 
The key observation in proving part (2) of Theorem \ref{thm 3+} is that an $n$-manifold with $n$-nonnegative curvature operator must be either flat or locally irreducible (see \cite[Theorem 1.8]{Li21}). This allows one to invoke the classification of simply-connected locally irreducible manifolds satisfying the property that $M\times \R$ has nonnegative isotropic curvature (see for example \cite{Brendle10} or \cite{Brendle10book}). The last step is to rule out K\"ahler manifolds by showing that a K\"ahler manifold with four-nonnegative curvature operator of the second kind must be flat (see \cite[Theorem 1.9]{Li21}). 

Another important result obtained by Cao, Gursky and Tran in \cite{CGT21} states that 
\begin{theorem}\label{thm 4 PIC}
A closed simply-connected Riemannian manifold of dimension $n\geq 4$ with four-positive curvature operator of the second kind is homeomorphic to the $n$-sphere. 
\end{theorem}

The proof of Theorem \ref{thm 4 PIC} relies on the observation that four-positive curvature operator of the second kind implies positive isotropic curvature.
Theorem \ref{thm 4 PIC} then follows immediately from the work of Micallef and Moore \cite{MM88}.
The curvature assumption in Theorem \ref{thm 4 PIC} is sharp in the sense that it cannot weakened to five-positive curvature operator of the second kind in general. Indeed, both the complex projective space $\mathbb{CP}^2$ and the cylinder $\mathbb{S}^3 \times \mathbb{S}^1$ has five-positive curvature operator of the second kind. 

In this paper, we prove a sharper result by weakening the assumption in Theorem \ref{thm 4 PIC} to $4\frac 1 2$-positive curvature operator of the second kind, whose meaning we explain below (see also Definitions \ref{def curvature operator of the second kind} and \ref{def curvature operator of the second kind on manifolds}).

Let $R$ be an algebraic curvature operator on a Euclidean vector space of dimension $n\geq 2$. Denote by $S^2_0(V)$ the space of traceless symmetric two-tensors on $V$ and $N=\dim(S^2_0(V))=\frac{(n-1)(n+2)}{2}$.
Let $\l_1 \leq \cdots \leq \l_N$ be the eigenvalues\footnote{The eigenvalues of $\mathring{R}$ are the eigenvalues of the matrix whose $(i,j)$ entry is $\mathring{R}(\vp_i,\vp_j)$, where $\{\vp_i\}_{i=1}^N$ is an orthonormal basis of $S^2_0(V)$. This definition is independent of the choices of the orthonormal basis.} of the curvature operator of the second kind $\mathring{R}$. We say $R$ has $4\frac 1 2$-positive curvature operator of the second kind if 
$$\l_1 +\l_2 +\l_3 +\l_4 +\frac 1 2 \l_5 >0.$$ 
Clearly, $4\frac 1 2$-positive curvature operator of the second kind is weaker than four-positive curvature operator of the second kind but stronger than five-positive curvature operator of the second kind.
More generally, for $1\leq k \leq N$ and $ 0 \leq  \a \leq  1$ satisfying $k+\a \leq N$, we say $R$ has $(k+\a)$-positive (respectively, $(k+\a)$-nonnegative) curvature operator of the second kind if 
$$\l_1 +\cdots +\l_k +\a  \l_{k+1} > ( \text{respectively,} \geq ) \  0.$$

The main result of this paper states the following. 
\begin{theorem}\label{thm 4.5 PIC}
A closed Riemannian manifold of dimension $n\geq 4$ with $4\frac 1 2$-positive curvature operator of the second kind is homeomorphic to a spherical space form. Moreover, if either $n=4$ or $n\geq 12$, then the ``homeomorphism" can be upgraded to ``diffeomorphism". 
\end{theorem}

We would like to point out that the curvature assumption in Theorem \ref{thm 4.5 PIC}  in general cannot be weakened to $(4+\a)$-positive curvature operator of the second kind for any $\a >\frac 1 2$. 
For the complex projective space $\mathbb{CP}^2$ with the Fubini-Study metric, the eigenvalues of $\mathring{R}$ are given by $\{-2,-2,-2,4,4,4,4,4,4\}$, up to scaling (see \cite{BK78} or \cite{CGT21}). For $\mathbb{S}^3 \times \mathbb{S}^1$ with the product metric, the eigenvalues of $\mathring{R}$ are given by $\{-\frac 1 2,0,0,0,1,1,1,1,1\}$, up to scaling (see \cite{Li21}). Therefore, both $\mathbb{CP}^2$ and $\mathbb{S}^3 \times \mathbb{S}^1$ have $(4+\a)$-positive curvature operator of the second kind for any $\a \in (\frac 1 2 ,1]$ but not for any $\a \in [0, \frac 1 2]$.
 
We also prove a rigidity result for $4\frac 1 2$-nonnegative curvature operator of the second kind. 
\begin{theorem}\label{thm 4.5 weakly PIC}
Let $(M^n,g)$ be a closed non-flat Riemannian manifold of dimension $n\geq 4$. 
Suppose that $M$ has $4\frac 1 2$-nonnegative curvature operator of the second kind. 
Then one of the following statements holds:
\begin{enumerate}
    \item $M$ is homeomorphic (diffeomorphic if $n=4$ or $n\geq 12$) to a spherical space form;
    \item $n=2m$ and the universal cover of $M$ is a K\"ahler manifold biholomorphic to $\mathbb{CP}^m$;
    \item $n=4$ and the universal cover of $M$ is diffeomorphic to $\mathbb{S}^3 \times \R$'; % or $M$ is a K\"ahler surface biholomorphic to $\mathbb{CP}^2$;
    \item $n\geq 5$ and $M$ is isometric to a quotient of a compact irreducible symmetric space\footnote{This case can be ruled out using \cite[Theorem B]{NPW22}.}.
\end{enumerate}
\end{theorem}

The key ingredients in proving Theorem \ref{thm 4.5 PIC} and Theorem \ref{thm 4.5 weakly PIC} are the following two results.
\begin{theorem}\label{thm 4.5 PIC alg}
Let $R$ be an algebraic curvature operator on a Euclidean vector space $V$ of dimension $n\geq 4$. If $R$ has $4\frac 1 2$-positive (respectively, $4\frac 1 2$-nonnegative) curvature operator of the second kind. Then $R$ has positive (respectively, nonnegative) isotropic curvature.
\end{theorem}

\begin{theorem}\label{thm Ric}
Let $R$ be an algebraic curvature operator on a Euclidean vector space $V$ of dimension $n\geq 3$. If $R$ has $(n+\frac {n-2}{n})$-positive (respectively, $(n+\frac {n-2}{n})$-nonnegative) curvature operator of the second kind, then $R$ has positive (respectively, nonnegative) Ricci curvature.
\end{theorem}

In dimension four, the assumption in Theorem \ref{thm 4.5 PIC alg} cannot be weakened to $(4+\a)$-positive (respectively, $(4+\a)$-nonnegative) curvature operator of the second kind for any $\a > \frac 1 2$ in view of the fact that $\mathbb{CP}^2$ has $4\frac 1 2$-nonnegative (but not $4\frac 1 2$-positive) curvature operator of the second kind and nonnegative (but not positive) isotropic curvature. 
However, it is unclear in higher dimensions. 

The assumption in Theorem \ref{thm Ric} is sharp in all dimensions, as  $\mathbb{S}^{n-1} \times \mathbb{S}^1$ has $(n+\frac {n-2}{n} )$-nonnegative (but not $(n+\frac {n-2}{n})$-positive) curvature operator of the second kind (see \cite[Example 2.5]{Li21}).

In dimension three, Theorem \ref{thm Ric} says that $3\frac 1 3$-positive (respectively, $3\frac 1 3$-nonnegative) curvature operator of the second kind implies positive (respectively, nonnegative) Ricci curvature. Using Hamilton's classification of closed three-manifolds with positive or nonnegative Ricci curvature in \cite{Hamilton82, Hamilton86}, we obtain the following theorem for three-manifolds, which improves Theorem \ref{thm 3+}.

\begin{theorem}\label{thm 3.3 3D}
Let $(M^3,g)$ be a three-dimensional closed Riemannian manifold. If $M$ has $3\frac 1 3$-positive curvature operator of the second kind, then $M$ is diffeomorphic to $\mathbb{S}^3$ or its quotient. If $M$ has $3\frac 1 3$-nonnegative curvature operator of the second kind, then $M$ 
is diffeomorphic to a quotient of one of the spaces $\mathbb{S}^3$ or $\mathbb{S}^2 \times \R$ or $\mathbb{R}^3$ by a group
of fixed point free isometries in the standard metrics.
\end{theorem}

The referee kindly pointed out that, by appealing to the result of Liu \cite{Liu13} on the classification of complete noncompact three-manifolds with nonnegative Ricci curvature, we have the following theorem. 
\begin{theorem}\label{thm 3.3 Liu}
Let $(M^3,g)$ be a complete noncompact three-manifold with $3\frac 1 3$-nonnegative curvature operator of the second kind. Then either $M$ is diffeomorphic to $\R^3$ or the universal cover of $M$ is isometric to a Riemann product $N^2 \times \R$, where $N^2$ is a complete 2-manifold with nonnegative sectional curvature.
\end{theorem}

It was shown by Micallef and Wang \cite{MW93} for $n=4$ and Brendle \cite{Brendle10} for $n\geq 5$ that Einstein manifolds with positive isotropic curvature have constant sectional curvature and closed Einstein manifolds with nonnegative isotropic curvature are locally symmetric. 
Thus Theorem \ref{thm 4.5 PIC alg} implies the following theorem for Einstein manifolds.
\begin{theorem}\label{thm Einstein}
Let $(M^n,g)$ be an Einstein manifold of dimension $n\geq 4$. If $M$ has $4\frac 1 2$-positive curvature operator of the second kind, then $M$ has constant sectional curvature. If $M$ is closed and has $4\frac 1 2$-nonnegative curvature operator of the second kind, then $M$ is locally symmetric. 
\end{theorem}

Finally, we propose the following conjecture.
\begin{conjecture}\label{conj n}
A closed $n$-dimensional Riemannian manifold with $(n+\frac{n-2}{n})$-positive curvature operator of the second kind is diffeomorphic to a spherical space form.
\end{conjecture}

When $n=2$, it is easy to see that two-positive curvature operator of the second is equivalent to positive scalar curvature. So Conjecture \ref{conj n} holds in dimension two by the uniformization theorem. 
What we prove in this paper is that Conjecture \ref{conj n} holds in three and four dimensions in Theorem \ref{thm 3.3 3D} and Theorem \ref{thm 4.5 PIC}, respectively. 
Note that the assumption cannot be weakened to $(n+\a)$-positive curvature operator of the second kind for any $\a > \frac{n-2}{n}$ as the product manifold $\mathbb{S}^{n-1}\times \mathbb{S}^1$ has $(n+\a)$-positive curvature operator of the second kind for any $\a > \frac{n-2}{n}$ (see \cite[Example 2.5]{Li21}). 

In a subsequent work \cite{Li22PAMS}, the author proves that a closed K\"ahler surface with six-positive curvature operator of the second kind is biholomorphic to $\mathbb{CP}^2$, and a closed nonflat K\"ahler surface with six-nonnegative curvature operator of the second kind is either biholomorphic to $\mathbb{CP}^2$ or isometric to $\mathbb{S}^2 \times \mathbb{S}^2$. This is done by showing that six-positive (respectively, six-nonnegative) curvature operator of the second kind implies positive (respectively, nonnegative) orthogonal bisectional curvature for K\"ahler manifolds.

The paper is organized as follows. 
In Section 2, we give an introduction to the curvature operator of the second kind.
In Section 3, we prove Theorem \ref{thm 4.5 PIC alg}. 
In Section 4, we prove Theorem \ref{thm Ric}. 
The proofs of Theorem \ref{thm 4.5 PIC} and Theorem \ref{thm 4.5 weakly PIC} are presented in Section 5.

\section{Curvature operator of the second kind}

The purpose of this section is to give an introduction to curvature operator of the second kind, as it is significantly less investigated than curvature operator (of the first kind) in the literature. 
Along the way, we also fix our notations, state our conventions and make some definitions. 
%Most of this section is extracted from \cite{Li21}. 
The reader is also referred to \cite{OT79, BK78, Nishikawa86, Kashiwada93, CGT21} for more information as well as previous results concerning the curvature operator of the second kind. 

Let $(V,g)$ be a (real) Euclidean vector space of dimension $n\geq 2$.  We always identify $V$ with its dual space $V^*$ via the metric $g$. 
%Throughout the discussion, $\{e_1, \cdots, e_n \}$ is an orthonormal basis of $V$. 
The space of bilinear forms on $V$ is denote by $T^2(V)$, and it splits as 
\begin{equation*}
    T^2(V)=S^2(V)\oplus  \Lambda^2(V),
\end{equation*}
where $S^2(V)$ is the space of symmetric two-tensors on $V$ and $\Lambda^2(V)$ is the space of two-forms on $V$. 
Our conventions on symmetric products and wedge products are that, for $u$ and $v$ in $V$, $\odot$ denotes the symmetric product defined by 
\begin{equation*}
    u \odot v  =u \otimes v + v \otimes u,
\end{equation*}
and 
$\wedge$ denotes the wedge product defined by 
\begin{equation*}
    u \wedge v  =u \otimes v -v \otimes u.
\end{equation*}
The inner product $g$ on $V$ naturally induces inner products on $S^2(V)$ and $\Lambda^2(V)$. 
To be consistent with \cite{CGT21}, the inner product on $S^2(V)$ is defined as 
\begin{equation*}
    \langle A, B \rangle =\tr(A^T B),
\end{equation*}
and the inner product on $\Lambda^2(V)$ is defined as 
\begin{equation*}
    \langle A, B \rangle =\frac{1}{2}\tr(A^T B).
\end{equation*}
In particular, if $\{e_1, \cdots, e_n\}$ is an orthonormal basis for $V$, then 
$\{e_i \wedge e_j\}_{1\leq i < j \leq n}$ is an orthonormal basis for $\Lambda^2(V)$ and $\{\frac{1}{\sqrt{2}}e_i \odot e_j\}_{1\leq i < j \leq n} \cup \{\frac{1}{2} e_i \odot e_i\}_{1\leq i\leq n}$ is an orthonormal basis for $S^2(V)$.

The space of symmetric two-tensors on $\Lambda^2(V)$ has the orthogonal decomposition 
\begin{equation*}
    S^2(\Lambda^2 (V))=S^2_B(\Lambda^2 (V)) \oplus \Lambda^4 (V),
\end{equation*}
where $S^2_B(\Lambda^2 (V))$ consists of all tensors $R\in S^2(\Lambda^2 (V))$ that also satisfy the first Bianchi identity. 
Any $R\in S^2_B(\Lambda^2 (V))$ is called an algebraic curvature operator. 

By the symmetries of $R\in  S^2_B(\Lambda^2 (V))$ (not including the first Bianchi identity), there are (up to sign) two ways that $R$ can induce a symmetric linear map $R:T^2(V) \to T^2(V)$. 
The first one, denoted by $\hat{R}: \Lambda^2(V) \to \Lambda^2(V)$ in this paper, is the so-called curvature operator defined by
\begin{equation}\label{eq R hat}
    \hat{R}(e_i\wedge e_j) =\frac 1 2 \sum_{k,l}R_{ijkl}e_k \wedge e_l
\end{equation}
where $\{e_1, \cdots, e_n\}$ is an orthonormal basis of $V$. 
Note that if the eigenvalues of $\hat{R}$ are all greater than or equal to $\kappa \in \R$, then all the sectional curvatures of $R$ are bounded from below by $\kappa$.  

The second one, denoted by $\mathring{R}:S^2(V) \to S^2(V)$, is defined by 
\begin{equation}\label{eq R ring}
    \mathring{R}(e_i \odot e_j) =\sum_{k,l}R_{iklj} e_k \odot e_l.
\end{equation}

However, on contrary to the case of $\hat{R}$, all eigenvalues of $\mathring{R}$ being nonnegative implies all the sectional curvatures of $R$ are zero, that it, $R\equiv 0$. 
The new feature here is that $S^2(V)$ is not irreducible under the action of the orthogonal group $O(V)$ of $V$. 
The space $S^2(V)$ splits into $O(V)$-irreducible subspaces as 
\begin{equation*}
    S^2(V)=S^2_0(V) \oplus \R g,
\end{equation*}
where $S^2_0(V)$ denotes the space of traceless symmetric two-tensors on $V$.
The map $\mathring{R}$ defined in \eqref{eq R ring} then induces a bilinear form
\begin{equation}\label{eq 2.3}
    \mathring{R}:S^2_0(V) \times S^2_0(V) \to \R
\end{equation}
by restriction to $S^2_0(V)$. 
Note that if all the eigenvalues of $\mathring{R}$ restricted to $S^2_0(V)$ are bounded from by $\kappa \in \R$, then the sectional curvatures of $R$ are bounded from below by $\kappa$. 
It should be noted that $\mathring{R}$ does not preserve the subspace $S^2_0(V)$ in general, but it does, for instance, when $R$ has constant Ricci curvature. 

Following \cite{Nishikawa86}, we call $\mathring{R}$ in \eqref{eq 2.3} (restricted to $S^2_0(V)$) the \textit{curvature operator of the second kind}, to distinguish it from the map $\hat{R}$ defined in \eqref{eq R hat}, which he called the \textit{curvature operator of the first kind}.

We make the following definitions. 
\begin{definition}\label{def curvature operator of the second kind}
Let $R\in S^2_B(\Lambda^2(V))$ be an algebraic curvature operator. Denote by $\l_1 \leq \cdots \leq \l_N$ the eigenvalues of $\mathring{R}$ (restricted to $S^2_0(V)$), where $N=\dim(S^2_0(V))=\frac{(n-1)(n+2)}{2}$. 
For $1\leq k \leq N$ and $0 \leq \a \leq 1$ satisfying $k+\a \leq N$, we say 
$R$ has $(k+\a)$-nonnegative curvature operator of the second kind if 
\begin{equation*}
    \l_1 + \cdots +\l_k + \a \l_{k+1} \geq 0. 
\end{equation*}
If the above inequality is strict, $R$ is said is to have $(k+\a)$-positive curvature operator of the second kind. 
\end{definition}

\begin{definition}\label{def curvature operator of the second kind on manifolds}
Let $(M^n,g)$ be a Riemannian manifold of dimension $n$. 
For $1\leq k \leq N$ and $0 \leq \a \leq 1$ satisfying $k+\a \leq N$, we say
$M$ has $(k+\a)$-positive (respectively, $(k+\a)$-nonnegative) curvature operator of the second kind if the Riemannian curvature tensor at each point $p\in M$ has $(k+\a)$-positive (respectively, $(k+\a)$-nonnegative) curvature of the second kind.
\end{definition}

\section{Positive Isotropic Curvature}

The notion of positive isotropic curvature was introduced by Micallef and Moore \cite{MM88} in their study of minimal two-spheres in Riemannian manifolds. They proved that a simply-connected closed Riemannian manifold with positive isotropic curvature is homeomorphic to the sphere. 
We recall the following definition for the reader's convenience. 

\begin{definition}
Let $R \in S^2_B (\Lambda^2 V)$ be an algebraic curvature operator on a Euclidean vector space $V$ of dimension $n\geq 4$. 
We say $R$ has nonnegative isotropic curvature if for all orthonormal four-frames $\{e_1, e_2, e_3, e_4\}\subset V$, it holds that 
    \begin{equation*}
        R_{1313}+R_{1414}+R_{2323}+R_{2424}-2R_{1234} \geq 0.
    \end{equation*}
If the inequality is strict, $R$ is said to have positive isotropic curvature. 
\end{definition}

In this section, we prove Theorem \ref{thm 4.5 PIC alg}, which is restated below. 
\begin{theorem}
Suppose $n=\dim(V)\geq 4$ and $R\in S^2_B(\Lambda^2(V))$ has $4\frac 1 2$-positive (respectively, $4\frac 1 2$-nonnegative) curvature operator of the second kind. Then for any orthonormal four-frame $\{e_1, \cdots, e_4\} \subset V$, we have
\begin{equation}\label{eq PIC}
     R_{1313}+ R_{1414}+ R_{2323}+  R_{2424} - 2 R_{1234} > (\text{ respectively, } \geq   ) \ 0. 
\end{equation}
\end{theorem}

\begin{proof}
Let $\{e_1, e_2,e_3,e_4\}$ be an orthonormal four-frame in $V$. We define the following symmetric two-tensors on $V$:
\begin{eqnarray*}
\vp_1 &=& \frac{1}{2} (e_1\odot e_1+e_2\odot e_2 -e_3\odot e_3 -e_4 \odot e_4),  \\
\vp_2 &=& \frac{1}{2} (e_1\odot e_1-e_2\odot e_2 +e_3\odot e_3 -e_4 \odot e_4), \\
\vp_3 &=& \frac{1}{2} (e_1\odot e_1-e_2\odot e_2 -e_3\odot e_3 +e_4 \odot e_4), \\
\vp_4 &=& e_1 \odot e_4 +e_2\odot e_3, \\
\vp_5 &=& e_1 \odot e_4 -e_2\odot e_3, \\
\vp_6 &=& e_1 \odot e_3 +e_2\odot e_4, \\
\vp_7 &=& e_1 \odot e_3 -e_2\odot e_4, \\
\vp_8 &=& e_1 \odot e_2 +e_3\odot e_4, \\
\vp_9 &=& e_1 \odot e_2 -e_3\odot e_4. 
\end{eqnarray*}
It is easy to verify that these tensors are traceless, mutually orthogonal in $S^2_0(V)$ and of the same magnitude $2$. 

If $R$ has $4\frac 1 2$-nonnegative curvature operator of the second kind, we get
\begin{eqnarray*}
 && \mathring{R}(\vp_1,\vp_1)+\mathring{R}(\vp_5,\vp_5)+\mathring{R}(\vp_6,\vp_6)+\mathring{R}(\vp_2,\vp_2)+\frac{1}{2}\mathring{R}(\vp_3,\vp_3) \geq 0, \\
 && \mathring{R}(\vp_1,\vp_1)+\mathring{R}(\vp_5,\vp_5)+\mathring{R}(\vp_6,\vp_6)+\mathring{R}(\vp_3,\vp_3)+\frac{1}{2}\mathring{R}(\vp_2,\vp_2) \geq 0, \\
 && \mathring{R}(\vp_1,\vp_1)+\mathring{R}(\vp_5,\vp_5)+\mathring{R}(\vp_6,\vp_6)+\mathring{R}(\vp_4,\vp_4)+\frac{1}{2}\mathring{R}(\vp_7,\vp_7) \geq 0, \\
&& \mathring{R}(\vp_1,\vp_1)+\mathring{R}(\vp_5,\vp_5)+\mathring{R}(\vp_6,\vp_6)+\mathring{R}(\vp_7,\vp_7)+\frac{1}{2}\mathring{R}(\vp_4,\vp_4) \geq 0, \\
&& \mathring{R}(\vp_1,\vp_1)+\mathring{R}(\vp_5,\vp_5)+\mathring{R}(\vp_6,\vp_6)+\mathring{R}(\vp_8,\vp_8)+\frac{1}{2}\mathring{R}(\vp_9,\vp_9) \geq 0, \\
&& \mathring{R}(\vp_1,\vp_1)+\mathring{R}(\vp_5,\vp_5)+\mathring{R}(\vp_6,\vp_6)+\mathring{R}(\vp_9,\vp_9)+\frac{1}{2}\mathring{R}(\vp_8,\vp_8) \geq 0. 
\end{eqnarray*}
Adding the above six inequalities together yields
\begin{eqnarray}\label{eq PIC 6}
\nonumber && 6\left(\mathring{R}(\vp_1,\vp_1)+\mathring{R}(\vp_5,\vp_5)+\mathring{R}(\vp_6,\vp_6)\right) +\frac{3}{2}\left( \mathring{R}(\vp_2,\vp_2)+\mathring{R}(\vp_3,\vp_3)\right) \\
 && +\frac{3}{2}\left(\mathring{R}(\vp_4,\vp_4)+\mathring{R}(\vp_7,\vp_7) \right)
 +\frac{3}{2}\left(\mathring{R}(\vp_8,\vp_8)+\mathring{R}(\vp_9,\vp_9)\right) \geq 0 . 
\end{eqnarray}
On the other hand, direct calculation shows
\begin{eqnarray*}
\mathring{R}(\vp_1,\vp_1) &=& 2(-R_{1212}-R_{3434}+R_{1313}+R_{2424}+R_{1414}+R_{2323}),\\
\mathring{R}(\vp_2,\vp_2) &=&2(-R_{1313}-R_{2424}+R_{1212}+R_{3434}+R_{1414}+R_{2323}), \\
\mathring{R}(\vp_3,\vp_3) &=& 2(-R_{1414}-R_{2323}+R_{1212}+R_{3434}+R_{1313}+R_{2424}) ,
\end{eqnarray*}
and
\begin{eqnarray*}
\mathring{R}(\vp_4,\vp_4) &=& 2(R_{1414}+R_{2323}+2R_{1234}-2R_{1342} ), \\
\mathring{R}(\vp_5,\vp_5) &=& 2(R_{1414}+R_{2323}-2R_{1234}+2R_{1342} ), \\
\mathring{R}(\vp_6,\vp_6) &=& 2(R_{1313}+R_{2424}-2R_{1234}+2R_{1423} ), \\
\mathring{R}(\vp_7,\vp_7) &=& 2(R_{1313}+R_{2424}+2R_{1234}-2R_{1423}), \\
\mathring{R}(\vp_8,\vp_8) &=& 2(R_{1212}+R_{3434}+2R_{1423}-2R_{1342}),  \\
\mathring{R}(\vp_9,\vp_9) &=& 2(R_{1212}+R_{3434}-2R_{1423}+2R_{1342} ). 
\end{eqnarray*}
Therefore, we obtain
\begin{eqnarray*}
 && \mathring{R}(\vp_1,\vp_1)+\mathring{R}(\vp_5,\vp_5)+\mathring{R}(\vp_6,\vp_6) \\
 &=& 4(R_{1313}+R_{1414}+R_{2323}+R_{2424} )-2(R_{1212}+R_{3434})-12R_{1234} ,
\end{eqnarray*}
and 
\begin{eqnarray*}
 \mathring{R}(\vp_2,\vp_2)+\mathring{R}(\vp_3,\vp_3) & =& 4(R_{1212}+R_{3434}), \\
 \mathring{R}(\vp_8,\vp_8)+\mathring{R}(\vp_9,\vp_9) & =& 4(R_{1212}+R_{3434}), \\
 \mathring{R}(\vp_4,\vp_4)+\mathring{R}(\vp_7,\vp_7) & =& 2(R_{1313}+R_{1414}+R_{2323}+R_{2424})+12R_{1234}.
\end{eqnarray*}
Substituting the above four identities into \eqref{eq PIC 6} produces
\begin{equation*}
    27(R_{1313}+R_{1414}+R_{2323}+R_{2424})-54R_{1234} \geq 0.
\end{equation*}
Hence $R$ has nonnegative isotropic curvature. Similarly, if $R$ has $4\frac 1 2$-positive curvature operator of the second kind, then all the above inequalities become strict and we conclude that $R$ has positive isotropic curvature. 

\end{proof}

\begin{remark}
On the complex projective space $\mathbb{CP}^2$ with its Fubini-Study metric satisfying $\Ric=6g$, all the inequalities in the above proof become identities for a suitably chosen orthonormal frame. 
More precisely, if we choose the orthonormal frame $\{e_1,e_2, e_3, e_4\}$ such that the two-plane spanned by $e_1$ and $e_2$ maximizes the sectional curvature among all two-planes, then all the nonzero components of curvature tensors are given by 
\begin{eqnarray*}
&& R_{1212}=R_{3434}=4,\\
&& R_{1313}=R_{2424}=1,\\
&& R_{1414}=R_{2323}=1, \\
&& R_{1234}=R_{1423}=1, \\
&& R_{1342}=-1,
\end{eqnarray*}
and their variants by symmetry. 
Therefore, we have that $\mathring{R}(\vp_i,\vp_i)=-8$ for $i=1, 5, 6$ and $\mathring{R}(\vp_i,\vp_i)=16$ for $i=2, 3, 4, 7, 8, 9$. 
\end{remark}

We are ready to prove Theorem \ref{thm Einstein}. 
\begin{proof}[Proof of Theorem \ref{thm Einstein}]
As explained in the Introduction, it follows from Theorem \ref{thm 4.5 PIC alg} and the result of Micallef and Wang \cite{MW93} when $n=4$ and that of Brendle \cite{Brendle10} when $n\geq 5$. 
\end{proof}

\section{Positive Ricci Curvature}
In this section, we prove Theorem \ref{thm Ric}. 

\begin{proof}[Proof of Theorem \ref{thm Ric}]
Let $\{e_1,\cdots, e_n\}$ be an orthonormal basis of $V$. We define the following symmetric two-tensors on $V$:
\begin{eqnarray*}
\vp_1 &=& \frac{1}{2\sqrt{n(n-1)}} \left( (n-1) e_1\odot e_1-\sum_{p=2}^n e_p\odot e_p \right),  \\
\vp_i &=& \frac{1}{\sqrt{2}} e_1\odot e_i \text{ for } 2 \leq i \leq n, \\
\psi_{kl} &=& \frac{1}{\sqrt{2}}e_k\odot e_l \text{ for } 2 \leq k < l \leq n,  \\
\xi_j &=& \frac{1}{2\sqrt{j(j-1)}} \left(\sum_{p=2}^j e_p \odot e_p - (j-1)e_{j+1}\odot e_{j+1} \right), \text{ for } 2 \leq j \leq n-1 . 
\end{eqnarray*}
One easily verifies that these tensors form an orthonormal basis of $S^2_0(V)$. 

If $R$ has $(n+\frac{n-2}{n})$-nonnegative curvature operator of the second kind, then we get that 
for $2 \leq k < l \leq n$, 
\begin{eqnarray*}
 && \sum_{i=1}^n \mathring{R}(\vp_i,\vp_i)+\frac {n-2}{n} \mathring{R}(\psi_{kl},\psi_{kl}) \geq 0, 
\end{eqnarray*}
and for $2 \leq j \leq n$, 
\begin{eqnarray*}
&& \sum_{i=1}^n \mathring{R}(\vp_i,\vp_i)+\frac {n-2}{n} \mathring{R}(\xi_{j},\xi_{j}) \geq 0. 
\end{eqnarray*}
Adding the above $\frac{(n-2)(n+1)}{2}$-many inequalities together yields
\begin{eqnarray}\label{eq Ric}
\nonumber && \frac{(n-2)(n+1)}{2}  \sum_{i=1}^n \mathring{R}(\vp_i,\vp_i) +\frac {n-2}{n} \sum_{2\leq k< l \leq n} \mathring{R}(\psi_{kl},\psi_{kl}) \\
&&+  \frac {n-2}{n} \sum_{j=2}^{n-1} \mathring{R}(\xi_j,\xi_j) \geq 0.  
\end{eqnarray}
On the other hand, direct calculation 
using 
$(\vp_1)_{11}=\frac{\sqrt{n-1}}{\sqrt{n}}$ and
$(\vp_1)_{jj} = \frac{-1}{\sqrt{n(n-1)}}$ for  $2\leq j \leq n$
shows that
\begin{eqnarray}\label{eq 1}
\nonumber \mathring{R}(\vp_1,\vp_1) &=& \sum_{i,j=1}^n R_{ijji}(\vp_1)_{ii}(\vp_1)_{jj} \\ \nonumber
    &=& \frac{2}{n}\sum_{j=2}^n R_{1j1j} -\frac{1}{n(n-1)}\sum_{i,j=2}^{n} R_{ijij} \\ \nonumber
    %&=& \frac{1}{n(n-1)}\left( 2(n-1)\sum_{j=2}^n R_{1j1j} -\sum_{i,j=1}^{n-1} R_{ijij} \right) \\
    &=& \frac{2}{n}R_{11}-\frac{1}{n(n-1)}\left(S-2R_{11} \right)   \\
&=& \frac{2}{n-1} R_{11}-\frac{1}{n(n-1)}S,
\end{eqnarray}
where $R_{11}=\Ric(e_1,e_1)$ and $S$ denotes the scalar curvature.
We also calculate that 
\begin{eqnarray}\label{eq 2}
\sum_{i=2}^n \mathring{R}(\vp_i,\vp_i) = \sum_{i=2}^n R_{1i1i} = R_{11}
\end{eqnarray}
and 
\begin{eqnarray}\label{eq 3}
\sum_{2\leq k < l \leq n} \mathring{R}(\psi_{kl},\psi_{kl}) =\sum_{2\leq k < l \leq n} R_{klkl} =\frac{S}{2} - R_{11}.
\end{eqnarray}
Finally, we compute that
\begin{eqnarray}\label{eq 4}
\nonumber \sum_{j=2}^{n-1} \mathring{R}(\xi_j,\xi_j) 
\nonumber &=& \sum_{j=2}^{n-1} \sum_{p,q=1}^n R_{pqqp}(\xi_j)_{pp}(\xi_j)_{qq} \\
\nonumber &=& \sum_{j=2}^{n-1}\left(  \frac 2 j \sum_{p=2}^j  R_{j+1,p,j+1,p} -\frac{1}{j(j-1)} \sum_{p, q=2}^j R_{pqpq}\right) \\
\nonumber &=& \frac{2}{n-1} \sum_{2 \leq p<q \leq n} R_{pqpq} \\
&=& \frac{1}{n-1} (S-2R_{11}).
\end{eqnarray}
Plugging the above four identities \eqref{eq 1}, \eqref{eq 2}, \eqref{eq 3} and \eqref{eq 4} into \eqref{eq Ric}, we obtain that
\begin{eqnarray*}
0&\leq & \frac{(n-2)(n+1)}{2} \left(\frac{2}{n-1} R_{11}-\frac{1}{n(n-1)}S  +R_{11}\right)  \\
&&+\frac {n-2}{n} \left(\frac{S}{2}-R_{11} \right)
+  \frac {n-2}{n(n-1)} (S-2R_{11})  \\
%&=& \left( \frac{(n-2)(n+1)^2}{2} \right) \\
&=& \frac{(n-2)(n+1)(n+2)}{2n}  R_{11}.
\end{eqnarray*}
Hence $R$ has nonnegative Ricci curvature. 
Similarly, if $R$ has $(n+\frac{n-2}{n})$-positive curvature operator of the second kind, then all the above inequalities become strict and we get that $R$ has positive Ricci curvature.
\end{proof}
\begin{remark}
On the product manifold $\mathbb{S}^{n-1} \times \mathbb{S}^1$, all the inequalities in the above proof become identities if the orthonormal frame is chosen such that $e_1 \in T_p \mathbb{S}^1$ and $e_2, \cdots, e_n \in T_q\mathbb{S}^{n-1}$. 
\end{remark}

We now give the proofs of Theorem \ref{thm 3.3 3D} and Theorem \ref{thm 3.3 Liu}.
\begin{proof}[Proof of Theorem \ref{thm 3.3 3D}]
As explained in the Introduction, it follows from Theorem \ref{thm Ric} and Hamilton's classification of closed three-manifolds with positive or nonnegative Ricci curvature in \cite{Hamilton82, Hamilton86}.
\end{proof}

\begin{proof}[Proof of Theorem \ref{thm 3.3 Liu}]
It follows immediately by combining Theorem \ref{thm Ric} and Liu's classification \cite{Liu13} of complete noncompact three-manifolds with nonnegative Ricci curvature.
\end{proof}

\section{Proofs}

We present the proofs of Theorem \ref{thm 4.5 PIC} and Theorem \ref{thm 4.5 weakly PIC} in this section. 

\begin{proof}[Proof of Theorem \ref{thm 4.5 PIC}]
By Theorem \ref{thm Ric}, $M$ has positive Ricci curvature. 
Thus the fundamental group of $M$ is finite and 
$\widetilde{M}$, the universal cover of $M$, is also a closed manifold. 
Since $\widetilde{M}$ also has $4\frac 1 2$-positive curvature operator of the second kind, it has positive isotropic curvature by Theorem \ref{thm 4.5 PIC alg}.  
By the work of Micallef and Moore \cite{MM88}, $\widetilde{M}$ is homeomorphic to $\mathbb{S}^n$. 
 
On the other hand, since $\widetilde{M}$ has trivial fundamental group, it does not contain nontrivial incompressible $(n-1)$-dimensional space forms.  
Therefore, we can use the work of Hamilton \cite{Hamilton97} if $n=4$ and that of Brendle \cite{Brendle19} if $n\geq 12$ to conclude that $\widetilde{M}$ must be diffeomorphic to $\mathbb{S}^n$.

\end{proof}

\begin{proof}[Proof of Theorem \ref{thm 4.5 weakly PIC}]
By Theorem \ref{thm Ric}, $M$ has nonnegative Ricci curvature. 
Let $\widetilde{M}$ be the universal cover of $M$. By the Cheeger-Gromoll splitting theorem, $\widetilde{M}$ is isometric to a product of the form 
$N^{n-k} \times \R^k$, where $N$ has positive Ricci curvature and is compact. 

If $n\geq 5$, we can use Theorem 1.8 in \cite{Li21} to conclude that $M$ is locally irreducible, which implies that $k=0$. 
Hence $\widetilde{M}$ is compact. 
Since $\widetilde{M}$ has $4\frac 1 2$-nonnegative curvature operator of the second kind, it also has nonnegative isotropic curvature by Theorem \ref{thm 4.5 PIC alg}. 
One can then invoke the classification of simply-connected locally irreducible closed Riemannian manifolds with nonnegative isotropic curvature (see \cite[Theorem 9.30]{Brendle10book} and \cite{Seshadri09}) to conclude that one of the following statement holds:
\begin{enumerate}
    \item $\widetilde{M}$ is homeomorphic to $\mathbb{S}^n$;
    \item $n=2m$ and $\widetilde{M}$ is a K\"ahler manifold biholomorphic to $\mathbb{CP}^m$;
    \item $n\geq 5$ and $\widetilde{M}$ is isometric to a compact irreducible symmetric space. 
\end{enumerate}
%By Theorem 1.9 in \cite{Li21}, case (2) cannot occur when $n\geq 5$. 
Similar as in the proof of Theorem \ref{thm 4.5 PIC}, the ``homeomorphism" is case (1) can be upgraded to a ``diffeomorphism" if $n\geq 12$ in view of Brendle's work \cite{Brendle19}. 

Finally, we treat the $n=4$ case by dividing it into two subcases. 

\textbf{Case A:} $M$ is locally irreducible. 
In this case, the above argument remains valid and we conclude that $\widetilde{M}$ is either diffeomorphic to $\mathbb{S}^4$ using Hamilton's work \cite{Hamilton97} or $M$ is a K\"ahler manifold biholomorphic to $\mathbb{CP}^2$ using the classification of closed four-manifolds with nonnegative isotropic curvature (see \cite[Theorem 4.10]{MW93}). 

\textbf{Case B:} $M$ is locally reducible. 
We observe that in this case $\widetilde{M}$ must be isometric to $N^3 \times \R$, as it cannot split as the product of two manifolds of dimension two by Theorem 5.1 in \cite{Li21}. 
Now we have that $N$ is a closed and simply-connected three-manifold which is locally irreducible. Since $N \times \R$ has nonnegative isotropic curvature, we conclude that $N$ must have nonnegative Ricci curvature. Thus $N$ is diffeomorphic to $\mathbb{S}^3$ by Hamilton's classification of closed three-manifolds with nonnegative Ricci curvature \cite{Hamilton86}. 
Hence $\widetilde{M}$ is diffeomorphic to $\mathbb{S}^3 \times \R$. 
This finishes the proof. 
\end{proof}

\section*{Acknowledgments}
The author thanks Professor Xiaodong Cao for his interest in this work and some helpful discussions. He is also grateful to the anonymous referee for catching some typos and errors, pointing out Theorem 1.8, and making comments and suggestions which improved the readability of this paper.

\section*{Funding and Competing interests}
The author's research is partially supported by a start-up grant at Wichita State University. 
The author has no other competing interests to declare that are relevant to the content of this article.

\bibliographystyle{alpha}
\bibliography{ref}
\end{document}